\documentclass[11pt]{amsart}
\usepackage{tikz-cd}
\usepackage[utf8]{inputenc}
\usepackage{dirtytalk}
\usepackage{amsfonts}
\usepackage{graphicx}
\usepackage{xcolor}
\usepackage{hyperref}
\hypersetup{
  colorlinks=true,
  linkcolor=magenta,
  citecolor=magenta,
  urlcolor=magenta
}
\usepackage{amsmath, amsthm, amssymb, dsfont}
\usepackage{epsfig}
\usepackage{array}
\usepackage{amsrefs}
\usepackage{enumitem}
\usepackage{amssymb}
\usepackage{graphicx}
\usepackage{amsxtra}
\usepackage{mathbbol}
\usepackage{comment}
\usepackage{etoolbox}
\usepackage{amstext}
\usepackage{mathtools}
 \usepackage[normalem]{ulem}
\usepackage{microtype}
\DeclareSymbolFontAlphabet{\amsmathbb}{AMSb}
\numberwithin{equation}{section}
\newtheorem*{theorem*}{Theorem}
\newtheorem{theorem}{Theorem}[section]
\newtheorem{lemma}[theorem]{Lemma}
\newtheorem{remark}[theorem]{Remark}
\newtheorem{corollary}[theorem]{Corollary}
\newtheorem{proposition}[theorem]{Proposition}
\newtheorem{definition}[theorem]{Definition}
\newtheorem{example}[theorem]{Example}

\newcommand{\Prob}{\operatorname{Prob}}

\newcommand{\wstar}{\mathrm{w}^{*}}
\newcommand{\SE}{\mathsf{S}_{E}}

\newcommand{\id}{\mathrm{id}}

\newcommand{\cst}{\ifmmode\mathrm{C}^*\else{$\mathrm{C}^*$}\fi}

\newtheorem{maintheorem}{Theorem}

\def\Prob{\mathrm{Prob}}

\allowdisplaybreaks

\title[Generalized Powers averaging]{Unique equivariant pseudo expectation and generalized powers averaging for crossed product}

\author[T. Amrutam]{Tattwamasi Amrutam}
\address{Institute of Mathematics of the Polish Academy of Sciences, ul. Sniadeckich 8, 00-656, Warszawa, Poland}
\email{tattwamasiamrutam@gmail.com}
\thanks{T.A. is supported by National Science Centre, Poland Sonata, grant number 2025/59/D/ST1/03117. F.K. is partially supported by a grant from the IMU-CDC and the Simons Foundation, and by a grant from IPM (No. 1404460315). Z.N. is partially supported by a grant from IPM (No. 1404370053).}

\author[F. Khosravi]{Fatemeh Khosravi}
\address{Department of Pure Mathematics\\
Faculty of Mathematics and Statistics\\
University of Isfahan\\
Isfahan 81746-73441\\
Iran, And School of Mathematics\\
Institute for Research in Fundamental Sciences (IPM)\\
P. O. Box 19395-5746\\
Tehran\\
Iran}
\email{f.khosravi@mcs.ui.ac.ir / f.khosravi@ipm.ir}

\author[Z. Naghavi]{Zahra Naghavi}
\address{Department of Mathematics, Faculty of Mathematical Sciences, Alzahra University, Tehran, Iran, and School of Mathematics, Institute for Research in Fundamental Sciences (IPM), P.O. Box: 19395-5746, Tehran, Iran
}
\email{naghavi.zahra@gmail.com / z.naghavi@alzahra.ac.ir}

\date{}

\begin{document}
\begin{abstract}
Let $\Gamma$ be a countable discrete group and let $A$ be a
unital $\Gamma$-simple $\Gamma$-$C^*$-algebra. We prove that
$A\rtimes_r\Gamma$ has the generalized Powers
averaging property if and only if the canonical map
$A\rtimes_r\Gamma\to I_\Gamma(A)$ is the unique
$\Gamma$-equivariant pseudo-expectation. We also show that
this averaging property lifts to
$I_\Gamma(A)\rtimes_r\Gamma$, with the averaging coefficients
still belonging to $A$. 
\end{abstract}
\maketitle
\section{Introduction}
\label{sec:introduction}

In~\cite{Powers1975}, Powers proved that the reduced group
$C^*$-algebra of the free group on two generators is simple and
has a unique tracial state. The main ingredient was an averaging
argument which showed that an element with zero canonical trace can be made small
in norm by taking a finite average of its conjugates by group
unitaries. For many years, extensions of this argument, often
formulated through the Powers property (or its variants) and related geometric
criteria, provided the main approach to proving $C^*$-simplicity.
We refer to Bekka--Cowling--de la Harpe~\cite{BekkaCowlingDeLaHarpe1994}
for such developments and to de la Harpe's
survey~\cite{DeLaHarpe2007} for the history of the subject.

A turning point came with the work of
Kalantar--Kennedy~\cite{KalantarKennedy2017}. They proved that a
discrete group $\Gamma$ is $C^*$-simple if and only if its action
on the Furstenberg boundary $\partial_F\Gamma$~\cite{glasner2006proximal} is topologically
free. On this topological space, topological freeness is equivalent to
freeness. This brought boundary dynamics into the study of
simplicity of reduced group $C^*$-algebras. Subsequently,
Haagerup~\cite{Haagerup2016} and Kennedy~\cite{Kennedy2020}
independently proved that $C^*$-simplicity is equivalent to
Powers averaging. Thus the averaging property that had originally
served as a method for proving simplicity became a
characterization of simplicity itself.

These results also belong to the broader study of averaging by
unitary conjugation, going back to Dixmier~\cite{Dixmier1949}.
Powers averaging restricts the conjugating unitaries to those
coming from the group. More recently, averaging appears in Robert's work on selfless
$C^*$-algebras~\cite{Robert2025}, where selflessness yields a uniform
Dixmier averaging property. Averaging also gives information about intermediate
$C^*$-algebras. Amrutam--Kalantar~\cite{AmrutamKalantar2020}
proved that, for a countable $C^*$-simple group $\Gamma$ and a
minimal compact $\Gamma$-space $X$, every intermediate
$C^*$-algebra between $C_r^*(\Gamma)$ and
$C(X)\rtimes_r\Gamma$ is simple. Such inclusions are
$C^*$-irreducible in the terminology of
R{\o}rdam~\cite{Rordam2023}.

It is natural to ask how these ideas extend to reduced crossed
products. For a minimal action of a countable discrete group $\Gamma$
on a compact Hausdorff space $X$, Kawabe~\cite{Kawabe2017} characterized simplicity of
$C(X)\rtimes_r\Gamma$ by freeness of the action on the spectrum
of $I_\Gamma(C(X))$. This is the corresponding generalized
Furstenberg boundary (see also the works of Naghavi~\cite{naghavi2020furstenberg} and Behrouzi-Naghavi ~\cite{behrouzi2025faithful}). For an arbitrary group, however, simplicity of
$C(X)\rtimes_r\Gamma$ cannot in general be characterized using
only scalar averages of group conjugations. Such averaging,
applied to the canonical copy of $C_r^*(\Gamma)$, would already
force $\Gamma$ to be $C^*$-simple~\cite{bryder2018reduced}. The coefficients in $C(X)$
must therefore enter the averaging procedure.

Amrutam--Ursu~\cite{AmrutamUrsu2022} (also see \cite{amrutam2023powers}) introduced generalized
probability measures with coefficients in $C(X)$ and proved
that the resulting generalized Powers averaging property
characterizes simplicity of $C(X)\rtimes_r\Gamma$ for minimal
actions of countable discrete groups. They also established a relative version of averaging
for equivariant inclusions of commutative coefficient algebras,
leading to simplicity of all intermediate $C^*$-algebras.

The noncommutative case is more subtle. For a unital $\Gamma$-simple $C^*$-algebra $A$, simplicity of
$A\rtimes_r\Gamma$ is characterized by faithfulness of every
equivariant pseudo-expectation~\cite{KS19}. Pseudo-expectations
were introduced by Pitts~\cite{Pitts2017} (and studied further
by Pitts--Zarikian~\cite{PittsZarikian2015}, and  Zarikian~\cites{zarikian2017unique, Zarikian2019}). In the equivariant setting, a
pseudo-expectation is a $\Gamma$-equivariant unital completely
positive map from $A\rtimes_r\Gamma$ to $I_\Gamma(A)$ which
fixes $A$ pointwise. Whether simplicity
alone forces uniqueness is a separate question (see also the discussion in~\cite{GU25}). We take this uniqueness property
as our starting point and ask whether it gives a
 generalized Powers averaging property.

We use the same form of averaging used in \cite{AmrutamUrsu2022}\footnote{Ozawa, in an unpublished note has shown that the generalized averaging appearing in \cite{AmrutamUrsu2022} is equivalent to intersection property for commutative crossed products.}, now allowing noncommuting
coefficients. We call the resulting property the
\emph{generalized Powers averaging property}
(abbreviated GPAP); it requires that the averaging can be carried out
for every $y\in A\rtimes_r\Gamma$ (see
Definition~\ref{def:ncpap}). Using the noncommutative boundary
theory developed by Kennedy--Schafhauser~\cite{KS19}, our first main result identifies
this property with uniqueness of the equivariant
pseudo-expectation.
\begin{maintheorem}
\label{thm:mainaveraging}
Let $\Gamma$ be a countable discrete group and let $A$ be a
unital $\Gamma$-simple $\Gamma$-$C^*$-algebra. Then the
following conditions are equivalent:
\begin{enumerate}[label=\textup{(\roman*)},leftmargin=*]
\item $A\rtimes_r\Gamma$ has GPAP.
\item The canonical map $E_A:A\rtimes_r\Gamma\to I_\Gamma(A)$
is the unique $\Gamma$-equivariant pseudo-expectation.
\item The canonical conditional expectation
$E_{I_\Gamma(A)}:I_\Gamma(A)\rtimes_r\Gamma\to I_\Gamma(A)$
is the unique $\Gamma$-equivariant conditional expectation.
\end{enumerate}
\end{maintheorem}
The passage to larger coefficient algebras is another part of
the problem in terms of relative averaging. In the commutative setting, this passage is
visible in~\cite{AmrutamKalantar2020}*{Lemma~2.1} and
\cite{AmrutamUrsu2022}*{Lemma~4.1}. The equivariant injective
envelope has the additional rigidity that allows us to lift the averaging scheme.
\begin{maintheorem}
\label{thm:mainlifting}
Let $\Gamma$ be a countable discrete group and let $A$ be a
unital $\Gamma$-simple $\Gamma$-$C^*$-algebra. Assume that
$A\rtimes_r\Gamma$ has GPAP. Then, for every
$y_1,\ldots,y_n\in I_\Gamma(A)\rtimes_r\Gamma$ and
$\varepsilon>0$, there exist $a_1,\ldots,a_m\in A$ and
$s_1,\ldots,s_m\in\Gamma$ such that
$\sum_{j=1}^m a_ja_j^*=1$ and
\[
\left\|\sum_{j=1}^m a_j\lambda_{s_j}
\bigl(y_i-E_{I_\Gamma(A)}(y_i)\bigr)
\lambda_{s_j}^*a_j^*\right\|<\varepsilon,~\forall 1\leq i\leq n.
\]
\end{maintheorem}
We also consider conditions under which simplicity itself
implies GPAP. The vanishing-obstruction condition of
Kennedy--Schafhauser~\cite{KS19} and the FC-hypercentrality of the group considered by
Geffen--Ursu~\cite{GU25} give two such conditions.
We denote these conditions by
\textup{(V)} and \textup{(F)}, respectively.
Under each of them, simplicity of $A\rtimes_r\Gamma$ is
equivalent to GPAP.  We also give an example where neither of these conditions is satisfied and yet GPAP holds. 

\subsection*{Organization of the Paper} In addition to this section, there are three other sections. In
Section~\ref{sec:setup}, we fix notation and prove some preparatory results.
Section~\ref{sec:unique-pseudo-expectation} proves Theorem~\ref{thm:mainaveraging} and Theorem~\ref{thm:mainlifting}. In Section~\ref{sec:simplicity}, we show that each of the conditions \textup{(F)} and
\textup{(V)} implies GPAP. Finally,  we finish with an explicit averaging
calculation for a crossed product satisfying neither of them.

\section{Notation and preliminaries}
\label{sec:setup}
Throughout this paper, let $\Gamma$ be a countable discrete group and $A$ be a unital $\Gamma$-simple $\Gamma$-$C^*$-algebra.  An action of $\Gamma$ on $A$ is a group homomorphism $\alpha$ from $\Gamma$ to the group of $*$-automorphisms on $A$. A $C^*$-algebra equipped with a $\Gamma$-action is called a $\Gamma$-$C^*$-algebra. We say that $A$ is $\Gamma$-simple if $\{0\}$ and $A$ are its only $\Gamma$-invariant closed two-sided ideals. Moreover, let $I_\Gamma(A)$ be the $\Gamma$-injective envelope of $A$ in the sense of Hamana~\cite{Hamana1985}. By the functoriality of the $\Gamma$-injective envelope, $\alpha$ extends to an action of $\Gamma$ on $I_\Gamma(A)$. We shall continue to denote the extended action by $\alpha$, whenever there is no risk of confusion. 

We briefly describe the construction of the crossed product for unital $C^*$-algebras in general and refer the reader to \cite{brown2008textrm} for more details. We start with $\pi:A \to \mathbb{B}(\mathcal{H})$, a faithful $*$-representation. We denote by $\ell^2(\Gamma,\mathcal{H})$, the space of square summable $\mathcal{H}$-valued functions on $\Gamma$, i.e.,
\[\ell^2(\Gamma,\mathcal{H})=\left\{\xi:\Gamma\to \mathcal{H}\text{ such that }\sum_{t\in\Gamma}\|\xi(t)\|_{\mathcal{H}}^2<\infty\right\}\] The group $\Gamma$ acts on $\ell^2(\Gamma,\mathcal{H})$ by left translation,
\[\lambda_s\xi(t):=\xi(s^{-1}t), \quad \xi \in \ell^2(\Gamma,\mathcal{H}), \ s,t \in \Gamma.\]
Let $\sigma$ be the $*$-representation $\sigma:A \to \mathbb{B}(\ell^2(\Gamma,\mathcal{H}))$ defined by \[\sigma(a)(\xi)(t):=\pi\bigl(\alpha_{t^{-1}}(a)\bigr)\xi(t), \quad a \in A,\]
where $\xi \in \ell^2(\Gamma,\mathcal{H})$, $t \in \Gamma$.
The reduced crossed product $C^*$-algebra $A\rtimes_r\Gamma$ is the norm closure of the $*$-subalgebra of $\mathbb{B}(\ell^2(\Gamma,\mathcal H))$ generated by $\sigma(A)$ and $\{\lambda_s:s\in\Gamma\}$. Moreover, $\lambda_s\sigma(a)\lambda_{s^{-1}}=\sigma(\alpha_s(a))$ for all $s\in\Gamma$ and $a\in A$. We identify $A$ with its image $\sigma(A)$.

The reduced crossed product $A\rtimes_r\Gamma$ also comes equipped with a $\Gamma$-equivariant canonical conditional expectation $E_A:A\rtimes_r\Gamma\to A$ given by 
\[E_A\left(a_s\lambda_s\right)=
\begin{cases}
0 & \text{if $s\ne e$,}\\
a_e & \text{if $s=e$.}
\end{cases}\]
This expectation is faithful by \cite{brown2008textrm}*{Proposition~4.1.9}. The same notation \(\lambda_g\) will be used for the canonical group unitaries in $I_\Gamma(A)\rtimes_r\Gamma$. Under the natural inclusion
 $A\rtimes_r\Gamma  \subseteq I_\Gamma(A)\rtimes_r\Gamma$,
these canonical unitaries are identified with one another. Conjugation by them defines an action of \(\Gamma\) on either crossed product,
$\widetilde\alpha_g(x):=\lambda_gx\lambda_g^*.$ Under the natural inclusion $A \rtimes_r \Gamma \subseteq I_\Gamma(A) \rtimes_r \Gamma$, it follows that $E_{I_\Gamma(A)}|_{A \rtimes_r \Gamma} = E_A$.

Motivated by Powers' averaging argument~\cite{Powers1975} and the commutative case (generalized probability measures with coefficients
in $C(X)$), we introduce the following. We use the same averaging formula as in \cite{AmrutamUrsu2022} with coefficients in a
noncommutative algebra. The coefficients and their group elements are an indexed family, so repetitions are allowed.

\begin{definition}
\label{def:prob}
Let $\Prob_f(\Gamma,A)$ denote the collection of finite indexed families
$\mu=((a_j,s_j))_{j=1}^m$, where $a_j\in A$, $s_j\in\Gamma$, and
$\sum_{j=1}^m a_ja_j^*=1$.
Associated with $\mu$ is the unital completely positive map
$T_\mu:A\rtimes_r\Gamma\to A\rtimes_r\Gamma$ given by
\[
T_\mu(y):=\sum_{j=1}^m
a_j\lambda_{s_j}y\lambda_{s_j}^*a_j^*.
\]
For a bounded linear functional $\varphi$, we write
$\varphi\mu=\varphi\circ T_\mu$.

If $A\subseteq B$ is a unital equivariant
inclusion, the same formula for $\mu\in\Prob_f(\Gamma,A)$ defines
a map on $B\rtimes_r\Gamma$. We have
$\Prob_f(\Gamma,A)\subseteq\Prob_f(\Gamma,B)$.
For $\varphi\in S(B\rtimes_r\Gamma)$, we write
$K_\varphi^{\Prob_f(\Gamma,A)}:=\overline{\{\varphi\mu:\mu\in\Prob_f(\Gamma,A)\}}^{\,w^*}$.
We shall be concerned with the inclusion $A\subseteq I_\Gamma(A)$.
\end{definition}

Repetitions cannot in general be combined into a single coefficient.
Even in the commutative setting, two terms with the same group element
need not define the same map as one term with a new coefficient.
See~\cite{AmrutamUrsu2022}*{Remark~2.4 and Example~2.9}.

\begin{remark}
\label{rem:convex-semigroup}
Let $A\subseteq B$ be a unital equivariant inclusion.
The maps associated with $\Prob_f(\Gamma,A)\subseteq\Prob_f(\Gamma,B)$ on
$B\rtimes_r\Gamma$ form convex families closed under composition.

Let $\mu=((a_i,s_i))_{i=1}^m$ and $\nu=((b_j,t_j))_{j=1}^n$
belong to $\Prob_f(\Gamma,B)$, and let $r\in[0,1]$. Concatenating
$((\sqrt r\,a_i,s_i))_{i=1}^m$ and
$((\sqrt{1-r}\,b_j,t_j))_{j=1}^n$ gives a family $\omega$ with
\[
\sum_i(\sqrt r\,a_i)(\sqrt r\,a_i)^*
+\sum_j(\sqrt{1-r}\,b_j)(\sqrt{1-r}\,b_j)^*
=r1+(1-r)1=1.
\]
Thus $T_\omega=rT_\mu+(1-r)T_\nu$, proving convexity. For composition, let
$\mu*\nu=((a_i\alpha_{s_i}(b_j),s_it_j))_{i,j}$.
Its coefficients satisfy
\[
\sum_{i,j}a_i\alpha_{s_i}(b_jb_j^*)a_i^*
=\sum_i a_i\alpha_{s_i}\left(\sum_jb_jb_j^*\right)a_i^*
=\sum_i a_ia_i^*=1.
\]
Moreover, for $y\in B\rtimes_r\Gamma$, covariance gives
\[
(T_\mu\circ T_\nu)(y)
=\sum_{i,j}a_i\alpha_{s_i}(b_j)\lambda_{s_it_j}
y\lambda_{s_it_j}^*\bigl(a_i\alpha_{s_i}(b_j)\bigr)^*=T_{\mu*\nu}(y).
\]
Hence the composite has coefficients in $B$. If the coefficients
of $\mu$ and $\nu$ lie in $A$, both of these elements still have
coefficients in $A$.
Finally, the family $((1,e))$ gives the identity map.

Now observe that the canonical conditional expectation is compatible with these maps. Let
$\mu=((a_j,s_j))_{j=1}^m\in\Prob_f(\Gamma,B)$. Since $E_B$ is $B$-bimodular and
$\Gamma$-equivariant, we see that
$E_B(T_\mu(y))=\sum_j a_j\alpha_{s_j}(E_B(y))a_j^*=T_\mu(E_B(y))$
for any $y\in B\rtimes_r\Gamma$,
that is, $E_B\circ T_\mu=(T_\mu|_B)\circ E_B$. In particular
$T_\mu(\ker E_B)\subseteq\ker E_B$, and the same conclusion holds for
$\mu\in\Prob_f(\Gamma,A)$.
\end{remark}

\begin{definition}
\label{def:ncpap}
Let $A\subseteq B$ be a unital $\Gamma$-equivariant inclusion.
We say that $B\rtimes_r\Gamma$ has
\begin{enumerate}[label=\textup{(\alph*)},leftmargin=*]
\item the \emph{generalized Powers averaging property} if, for every $y\in B\rtimes_r\Gamma$,
$\inf_{\mu\in\Prob_f(\Gamma,B)}\|T_\mu(y-E_B(y))\|=0$;
\item \emph{relative GPAP with coefficients in $A$} if, for every
$y\in B\rtimes_r\Gamma$, we have
$\inf_{\mu\in\Prob_f(\Gamma,A)}\|T_\mu(y-E_B(y))\|=0$.
\end{enumerate}
In the relative property, the averaging coefficients are required
to lie in $A$, while $y$ may lie in $B\rtimes_r\Gamma$.
\end{definition}
We use the following convention. For $B=I_\Gamma(A)$, we refer to these two properties as \emph{GPAP upstairs} and \emph{relative
GPAP upstairs}, respectively. By \emph{GPAP downstairs}, we mean GPAP for
$A\rtimes_r\Gamma$. The following elementary observation gives
both the kernel formulation and the finite-set formulation of
these properties.

\begin{lemma}
\label{lem:centered-finite-averaging}
Let $A\subseteq B$ be a unital equivariant inclusion. Then the following are equivalent:
\begin{enumerate}[label=\textup{(\roman*)},leftmargin=*]
\item For every $y\in B\rtimes_r\Gamma$ and every $\varepsilon>0$,
there exists $\mu\in\Prob_f(\Gamma,A)$ such that $\|T_\mu(y-E_B(y))\|<\varepsilon$.
\item For every $x\in\ker E_B$, $\inf_{\mu\in\Prob_f(\Gamma,A)}\|T_\mu(x)\|=0$.
\item For every finite set $x_1,\ldots,x_n\in\ker E_B$ and every
$\varepsilon>0$, there exists $\mu\in\Prob_f(\Gamma,A)$ such that
$\|T_\mu(x_i)\|<\varepsilon$ for every $1\leq i\leq n$.
\item For every $y_1,\ldots,y_n\in B\rtimes_r\Gamma$ and every
$\varepsilon>0$, there exists $\mu\in\Prob_f(\Gamma,A)$ such that
$\|T_\mu(y_i-E_B(y_i))\|<\varepsilon$ for every $1\leq i\leq n$.
\end{enumerate}
The same equivalence holds with $\Prob_f(\Gamma,B)$ in place of
$\Prob_f(\Gamma,A)$.

\end{lemma}
\begin{proof}
Every $y-E_B(y)$ belongs to $\ker E_B$, and every $x\in\ker E_B$
equals $x-E_B(x)$. This proves \textup{(i)}$\Leftrightarrow$\textup{(ii)}.
Let us now prove \textup{(ii)}$\Rightarrow$\textup{(iii)}. Let
$x_1,\ldots,x_n\in\ker E_B$ and $\varepsilon>0$.
We construct $\nu_0,\ldots,\nu_n\in\Prob_f(\Gamma,A)$ such that
$\|T_{\nu_k}(x_i)\|<\varepsilon$ whenever $1\leq i\leq k$. Let $\nu_0=((1,e))$. Suppose that $\nu_{k-1}$ has been chosen.
By Remark~\ref{rem:convex-semigroup},
$T_{\nu_{k-1}}(x_k)\in\ker E_B$. Applying \textup{(ii)} to this
element, choose $\mu_k\in\Prob_f(\Gamma,A)$ such that
$\|T_{\mu_k}(T_{\nu_{k-1}}(x_k))\|<\varepsilon$.
Let $\nu_k=\mu_k*\nu_{k-1}$. By
Remark~\ref{rem:convex-semigroup}, $\nu_k\in\Prob_f(\Gamma,A)$ and
$T_{\nu_k}=T_{\mu_k}\circ T_{\nu_{k-1}}$.
Moreover, for every $1\leq i<k$, contractivity of $T_{\mu_k}$ gives
\[
\|T_{\nu_k}(x_i)\|
=\|T_{\mu_k}(T_{\nu_{k-1}}(x_i))\|
\leq\|T_{\nu_{k-1}}(x_i)\|
<\varepsilon.
\]
Thus $\nu_k$ has the required property. Taking $\mu=\nu_n$
proves \textup{(iii)}. The implication
\textup{(iii)}$\Rightarrow$\textup{(ii)} follows by taking a singleton.
The implication \textup{(iii)}$\Rightarrow$\textup{(iv)} follows by writing
$x_i=y_i-E_B(y_i)$, and the converse follows by taking $y_i=x_i$.
\end{proof}
The elementary completely positive maps of
\cite{amrutam2025splitting} will be used twice, once downstairs and
once upstairs. We record here the normalization required by
Definition~\ref{def:prob}.
\begin{lemma}
\label{lem:elementary-in-T}
Let $\rho\in S(A)$. There is a net $(\mu_\lambda)_\lambda$ in
$\Prob_f(\Gamma,A)$ such that
\[
\bigl\|T_{\mu_\lambda}(a)-\rho(a)1\bigr\|\longrightarrow 0
\qquad\text{for every }a\in A.
\]
Here $T_{\mu_\lambda}$ may be computed on $B\rtimes_r\Gamma$ for any
unital equivariant inclusion $A\subseteq B$; the coefficients of
$\mu_\lambda$ lie in $A$, so that $\mu_\lambda\in\Prob_f(\Gamma,A)$ in
either case.
\end{lemma}

\begin{proof}
By \cite{amrutam2025splitting}*{Proposition~3.2} and its proof, there
are, for each $\lambda$, an element $a_\lambda\in A^{+}$, elements
$\tilde a_1,\ldots,\tilde a_m\in A$ and group elements
$t_1,\ldots,t_m\in\Gamma$ (all depending on $\lambda$) such that
\[
\sum_{i=1}^{m}\tilde a_i\,\alpha_{t_i}(a_\lambda^{2})\,\tilde a_i^{*}=1,
~
\Psi_\lambda(x):=\sum_{i=1}^{m}
\tilde a_i\,\alpha_{t_i}(a_\lambda xa_\lambda)\,\tilde a_i^{*}
~(x\in A),
\]
and $\Psi_\lambda(a)\to\rho(a)1$ in norm for every $a\in A$.
Set $c_i:=\tilde a_i\,\alpha_{t_i}(a_\lambda)\in A$. Since $a_\lambda$
is positive, $\alpha_{t_i}(a_\lambda)^{*}=\alpha_{t_i}(a_\lambda)$, and
therefore for $x\in A$, we have
\[
\sum_{i=1}^{m}c_ic_i^{*}
=\sum_{i=1}^{m}\tilde a_i\,\alpha_{t_i}(a_\lambda^{2})\,\tilde a_i^{*}=1,
\text{and }
\sum_{i=1}^{m}c_i\,\alpha_{t_i}(x)\,c_i^{*}=\Psi_\lambda(x).
\]
Hence $\mu_\lambda:=((c_i,t_i))_{i=1}^{m}\in\Prob_f(\Gamma,A)$.
For $a\in A$, covariance gives
\[
T_{\mu_\lambda}(a)=\sum_{i=1}^m c_i\lambda_{t_i}a\lambda_{t_i}^{*}c_i^{*}
=\sum_{i=1}^m c_i\alpha_{t_i}(a)c_i^{*}=\Psi_\lambda(a),
\]
and the same formula defines $T_{\mu_\lambda}$ on $B\rtimes_r\Gamma$
for any unital equivariant inclusion $A\subseteq B$.
\end{proof}
Recall that every $\Gamma$-equivariant u.c.p.~map
from a $\Gamma$-operator subsystem into $I_\Gamma(A)$ extends to the containing
$\Gamma$-operator system. Moreover, an equivariant u.c.p.~map $I_\Gamma(A)\to I_\Gamma(A)$ fixing
$A$ pointwise is the identity (see
\cite{GU25}*{Theorem~2.8 and Definition~2.9}).

We record the following well-known facts for later use;
they are \cite{GU25}*{Proposition~2.15}, and we include the argument for
convenience.

\begin{lemma}
\label{lem:envelope-gamma-simple}
Let $A$ be a unital $\Gamma$-simple $\Gamma$-$C^*$-algebra. Then the following hold.
\begin{enumerate}[label=\textup{(\roman*)},leftmargin=*]
\item $I_\Gamma(A)$ is $\Gamma$-simple.
\item $Z(A)$ is $\Gamma$-simple.
\end{enumerate}
In particular, $Z(I_\Gamma(A))$ is $\Gamma$-simple.
\end{lemma}
\begin{proof}
\textup{(i)} Let $J$ be a proper closed $\Gamma$-invariant ideal of $I_\Gamma(A)$. Since inclusion \(A\subseteq I_\Gamma(A)\) is unital and $A$ is $\Gamma$-simple, then $J\cap A=0$. Hence the quotient map $q:I_\Gamma(A)\longrightarrow I_\Gamma(A)/J $ is a $\Gamma$-equivariant $*$-homomorphism whose restriction to $A$ is a complete order embedding. By the equivariant essentiality of the inclusion $A\subseteq I_\Gamma(A),$ the map $q$ is a complete order embedding on all of $I_\Gamma(A)$. Therefore $J=\ker q=\{0\}.$ Thus $I_\Gamma(A)$ is $\Gamma$-simple. 

\textup{(ii)} Suppose that $J$ is a nonzero proper invariant ideal of $Z(A)$.
Choose a character $\chi$ of $Z(A)$ vanishing on $J$,
and extend it to a state $\varphi$ on $A$.
For $b\in A$ and $z\in J$, the Cauchy--Schwarz inequality gives
$\varphi(bz)=0$, since $\varphi(z^*z)=\chi(z^*z)=0$.
Thus $\varphi$ vanishes on the ideal
$\overline{\operatorname{span}}AJ$ of $A$.
This ideal is nonzero and invariant, but proper because
$\varphi(1)=1$, contradicting $\Gamma$-simplicity.
Applying \textup{(ii)} to $I_\Gamma(A)$, which is $\Gamma$-simple by \textup{(i)}, gives the last assertion.
\end{proof}
Now, let $A\subseteq B$ be a unital equivariant inclusion and write $E=E_B$.
For $\varphi\in S(B\rtimes_r\Gamma)$, Remark~\ref{rem:convex-semigroup}
shows that $\overline{\{\varphi\mu:\mu\in\Prob_f(\Gamma,A)\}}^{\,w^*}$ is convex,
weak$^*$-compact, and invariant under averaging. That is for every
$\psi\in K_\varphi^{\Prob_f(\Gamma,A)}$ and $\mu\in\Prob_f(\Gamma,A)$,
$\psi\mu\in K_\varphi^{\Prob_f(\Gamma,A)}$. Moreover, the set
$\SE:=\{\rho\circ E:\rho\in S(B)\}$
is weak$^*$-compact and invariant under these averages, since
$(\rho\circ E)\mu=(\rho\circ T_\mu|_B)\circ E\in\SE$
for $\rho\in S(B)$ and $\mu\in\Prob_f(\Gamma,A)$.
The following lemma and proposition also hold with $\Prob_f(\Gamma,B)$
in place of $\Prob_f(\Gamma,A)$.

The following lemma shows that, given finitely many states, one common net
of averaging maps makes each state converge to a state factoring through
the canonical conditional expectation.

\begin{lemma}
\label{lem:simultaneous-orbit}
Assume that $\overline{\{\varphi\mu:\mu\in\Prob_f(\Gamma,A)\}}^{\,w^*}\cap\SE\neq\varnothing$ for every
$\varphi\in S(B\rtimes_r \Gamma)$. For any $\varphi_1,\ldots,\varphi_n\in S(B\rtimes_r \Gamma)$,
there exists a net $(\mu_\nu)_\nu$ in $\Prob_f(\Gamma,A)$ and
$\rho_1,\ldots,\rho_n\in S(B)$ such that
$\varphi_j\mu_\nu\longrightarrow\rho_jE ~\text{weak}^*$ for $1\leq j\leq n.$
\end{lemma}
\begin{proof}
Consider the set 
\[ \overline{\{(\varphi_1\mu,\ldots,\varphi_n\mu):\mu\in\Prob_f(\Gamma,A)\}}^{\wstar}.
\]
 It is invariant under diagonal averaging. We argue by induction. Suppose that $(\psi_1,\ldots,\psi_n)$ belongs to this set and that
$\psi_1,\ldots,\psi_k\in\SE$. Apply the hypothesis to $\psi_{k+1}$,
and choose a net $(\mu_\nu)_\nu$ in $\Prob_f(\Gamma,A)$ with
$\psi_{k+1}\mu_\nu$ converging to a point of $\SE$.
Pass to a subnet along which the entire tuple converges.
The first $k$ limiting coordinates remain in $\SE$ by invariance
and closedness, and the $(k+1)$st now belongs to $\SE$.
Starting with $k=0$ and repeating gives a point in
$\overline{\{(\varphi_1\mu,\ldots,\varphi_n\mu):\mu\in\Prob_f(\Gamma,A)\}}^{\wstar}\cap\SE^n$.
The claim follows.
\end{proof}
We now obtain a norm
averaging using \cite{Murphy1990}*{Theorem~A.7} from this state convergence.
\begin{proposition}
\label{prop:orbit-intersection}
Let $A\subseteq B$ and $E=E_B$ be as above. The following conditions are equivalent:
\begin{enumerate}[label=\textup{(\roman*)},leftmargin=*]
\item For every $x\in\ker E$,
$\inf_{\mu\in\Prob_f(\Gamma,A)}\|T_\mu(x)\|=0$.
\item $K_\varphi^{\Prob_f(\Gamma,A)}\cap\SE\neq\varnothing$ for every $\varphi\in S(B\rtimes_r \Gamma)$.
\end{enumerate}
Moreover, if these equivalent conditions hold, then for every
$x_1,\ldots,x_n\in\ker E$ and $\varepsilon>0$, there exists
$\mu\in\Prob_f(\Gamma,A)$ such that $\|T_\mu(x_k)\|<\varepsilon$ for every
$1\leq k\leq n$. In that case there is also a single net
$(\mu_\nu)_\nu$ in $\Prob_f(\Gamma,A)$ with $\|T_{\mu_\nu}(x)\|\to0$ for every
$x\in\ker E$.
\end{proposition}
\begin{proof}
Assume \textup{(i)}. Given $x_1,\ldots,x_n\in\ker E$ and
$\varepsilon>0$, The argument proving
\textup{(ii)}$\Rightarrow$\textup{(iii)} in
Lemma~\ref{lem:centered-finite-averaging} applies to
the maps $T_\mu$, $\mu\in\Prob_f(\Gamma,A)$, since they are contractive, preserve
$\ker E$, and are closed under composition. Hence, for every
$x_1,\ldots,x_n\in\ker E$ and $\varepsilon>0$, there exists
$\mu\in\Prob_f(\Gamma,A)$ such that $\|T_\mu(x_j)\|<\varepsilon$
for every $1\leq j\leq n$. Directing the finite subsets of $\ker E$ by inclusion and the
$\epsilon$'s down to zero gives a net $(\mu_\nu)_\nu$
such that $\|T_{\mu_\nu}(x)\|\to0$ for every $x\in\ker E$.
For $\varphi\in S(B\rtimes_r \Gamma)$, any weak$^*$ cluster point $\eta$ of
$\varphi\mu_\nu$ annihilates $\ker E$. Thus
$\eta=(\eta|_B)\circ E\in\SE$, proving \textup{(ii)}.

Conversely, assume \textup{(ii)} and fix $x\in\ker E$.
Every bounded functional $f\in (B\rtimes_r \Gamma)^*$ is a linear combination
$f=\sum_{j=1}^r c_j\varphi_j$ of at most four states.
By Lemma~\ref{lem:simultaneous-orbit}, there is a common net
$(\mu_\nu)_\nu$ such that $\varphi_j\mu_\nu\to\rho_jE$ for all $j$.
Consequently, $f(T_{\mu_\nu}(x))\longrightarrow \sum_{j=1}^r c_j\rho_j(E(x))=0.$ An appeal to the Hahn Banach separation theorem proves \textup{(i)}.
The first paragraph proves the simultaneous conclusion.
\end{proof}

We close this section with an elementary density criterion for
state spaces, which will be used in the proof of
Theorem~\ref{thm:ncpap-full-envelope-lifting}. Given a von Neumann algebra $M$, let $M_{\mathrm{sa}}$ denote its
self-adjoint part.
\begin{lemma}
\label{lem:density}
Let $M$ be a von Neumann algebra, and let $F\subseteq S(M)$ be such that
\[
\sup_{\sigma\in F}\sigma(h)=\min\{r\in\mathbb R:h\leq r1\}
\]
for every $h\in M_{\mathrm{sa}}$. Then the convex hull of $F$ is
weak$^*$ dense in $S(M)$.
\end{lemma}
\begin{proof}
Every state is a hermitian functional, hence is determined by its
values on $M_{\mathrm{sa}}$; and since $M=M_{\mathrm{sa}}+iM_{\mathrm{sa}}$,
the weak$^*$ topology on $S(M)$ coincides with the topology of pointwise
convergence on $M_{\mathrm{sa}}$. We may therefore separate in the real
locally convex space $(M_{\mathrm{sa}})^*$ with its weak$^*$ topology,
whose continuous dual is $M_{\mathrm{sa}}$. Let $K$ be the weak$^*$ closed convex
hull of $F$, and suppose that $\sigma_0\in S(M)\setminus K$.
\cite{Murphy1990}*{Theorem~A.7} gives $h=h^*\in M$ with
\[
\sigma_0(h)>\sup_{\sigma\in K}\sigma(h)\ge\sup_{\sigma\in F}\sigma(h)
=\min\{r:h\leq r1\}.
\]
This is impossible, since $\sigma_0(h)\leq r$ whenever $h\leq r1$.
\end{proof}
\section{Pseudo-expectations and averaging}
\label{sec:unique-pseudo-expectation}
Pseudo-expectations were
introduced and studied in the nonequivariant setting by
Pitts~\cite{Pitts2017} and Pitts--Zarikian~\cite{PittsZarikian2015},
where their basic properties, including existence and the behaviour of
the associated inclusions, are established.
A $\Gamma$-equivariant pseudo-expectation is an equivariant unital completely positive map
$\Phi\colon A\rtimes_r\Gamma\to I_\Gamma(A)$ satisfying $\Phi|_A=\id_A$.
An example of such a map is $E_A$, followed by the inclusion
$A\subseteq I_\Gamma(A)$. Equivariant pseudo-expectations always
exist. Indeed, by $\Gamma$-injectivity of $I_\Gamma(A)$, the inclusion
$A\subseteq I_\Gamma(A)$ extends to an equivariant u.c.p.~map on
$A\rtimes_r\Gamma$, which fixes $A$ pointwise. 

For an operator system $C$, a matrix state of size $n$ is a
unital completely positive map $C\to M_n$.
Matrix convex combinations allow us to combine these maps at
different sizes. In~\cite{KS19}*{Section~7}, invariant matrix convex sets
over the matrix state space of $A$ are used to define boundaries and
relate them to equivariant pseudo-expectations. We apply this correspondence to a matrix
convex set obtained from an averaging orbit.

We begin by proving that uniqueness of the equivariant
pseudo-expectation implies GPAP downstairs. For an operator system \(C\), let
\(\operatorname{MS}_n(C)=\operatorname{UCP}(C,M_n)\), and we write
\(\operatorname{MS}(C)=(\operatorname{MS}_n(C))_{n\geq1}\).

\begin{theorem}
\label{thm:unique-pseudo-expectation}
\begingroup
\emergencystretch=1em
Assume that $E_A$ is the unique
$\Gamma$-equivariant pseudo-expectation. Then
$A\rtimes_r\Gamma$ has GPAP.
\par\endgroup
\end{theorem}
\begin{proof}
Write $E=E_A$. Towards a contradiction, assume that GPAP fails.
By Proposition~\ref{prop:orbit-intersection} with $B=A$, there exists
$\phi_0\in S(A\rtimes_r\Gamma)$ such that
$K_{\phi_0}^{\Prob_f(\Gamma,A)}\cap\SE=\varnothing$.

We first observe that the restriction of
$K_{\phi_0}^{\Prob_f(\Gamma,A)}$
to $A$ is all of $S(A)$.
Fix $\rho\in S(A)$. Lemma~\ref{lem:elementary-in-T} gives a net
$(\mu_\lambda)_\lambda$ in $\Prob_f(\Gamma,A)$ such that
$\|T_{\mu_\lambda}(a)-\rho(a)1\|\to0$ for every $a\in A$.
Hence $(\phi_0T_{\mu_\lambda})(a)\longrightarrow\rho(a)$
for every $a\in A$.

Weak$^*$ compactness of $S(A\rtimes_r\Gamma)$ gives a convergent
subnet, with limit
$\eta\in K_{\phi_0}^{\Prob_f(\Gamma,A)}$.
Its restriction to $A$ is $\rho$. Consequently, we see that
$S(A)=\left\{\eta|_A:
\eta\in K_{\phi_0}^{\Prob_f(\Gamma,A)}
\right\}$. For each
$\eta\in K_{\phi_0}^{\Prob_f(\Gamma,A)}$
whose restriction $\rho:=\eta|_A$ is pure, let
$(\pi_\eta,H_\eta,\xi_\eta)$ be the associated GNS representation.
Let $H_\rho:=\overline{\pi_\eta(A)\xi_\eta}$.
The representation
$\pi_\rho(a):=\pi_\eta(a)|_{H_\rho}$, $a\in A$, is the GNS
representation of the pure state $\rho$, and hence it is
irreducible. Let $P_\rho$ be the orthogonal projection of
$H_\eta$ onto $H_\rho$, and define
$\Theta_\eta\colon A\rtimes_r\Gamma\to B(H_\rho)$ by
$\Theta_\eta(b)=P_\rho\pi_\eta(b)|_{H_\rho}$.
This map is unital and completely positive, and
$\Theta_\eta|_A=\pi_\rho$.

For $n\geq1$, let $R_n^\eta$ consist of the maps
$\mu_{\eta,V}(b):=V^*\Theta_\eta(b)V$, where
$V\colon\mathbb C^n\to H_\rho$ is an isometry.
The family $R_n^\eta$ does not depend on the chosen GNS
realization, since the intertwining unitary transports the
isometries $V$ and leaves the resulting maps unchanged.
Let $\mathcal K=(\mathcal K_n)_{n\geq1}$ be the levelwise
weak$^*$ closed matrix convex hull in
$\operatorname{MS}(A\rtimes_r\Gamma)$ of all the maps in
$R_n^\eta$, with $n\geq1$ and $\eta$ ranging over the states
in $K_{\phi_0}^{\Prob_f(\Gamma,A)}$ whose restriction to $A$ is pure.
Let $r\colon\mathcal K\to\operatorname{MS}(A)$,
$r(\mu)=\mu|_A$, be the restriction map.

We claim that $r(\mathcal K)=\operatorname{MS}(A)$.
Each $\mathcal K_n$ is weak$^*$ compact, so $r(\mathcal K)$
is a closed matrix convex subset of $\operatorname{MS}(A)$.
The matrix extreme points of $\operatorname{MS}(A)$ are
precisely the maps $a\mapsto W^*\pi(a)W$, where
$\pi\colon A\to B(H)$ is irreducible and
$W\colon\mathbb C^n\to H$ is an isometry
(see \cite{webster1999krein}*{Example~2.3} together with
\cite{Arveson1969}*{Corollary~1.4.3}).
Fix such a map and choose a unit vector $\zeta\in H$.
The vector state
$\sigma(a)=\langle\pi(a)\zeta,\zeta\rangle$ is pure.
We can find
$\psi\in K_{\phi_0}^{\Prob_f(\Gamma,A)}$ with $\psi|_A=\sigma$.
Applying the preceding construction to $\psi$, the
representation $\pi_\sigma$ on $H_\sigma$ is unitarily
equivalent to $\pi$.
Transporting $W$ through this unitary equivalence gives an
isometry $V\colon\mathbb C^n\to H_\sigma$ satisfying
$V^*\Theta_\psi(a)V=W^*\pi(a)W$ for every $a\in A$.
Thus the given matrix extreme point belongs to
$r_n(R_n^\psi)\subseteq r_n(\mathcal K_n)$.
Hence $r(\mathcal K)$ contains every matrix extreme point of
$\operatorname{MS}(A)$, and
\cite{webster1999krein}*{Theorem~4.3} gives
$r(\mathcal K)=\operatorname{MS}(A)$.

We next prove that
$\mathcal K_1\subseteq K_{\phi_0}^{\Prob_f(\Gamma,A)}$.
Let $\mu_{\eta,V}\in R_1^\eta$ and let $\xi=V(1)$.
Since $V\colon\mathbb C\to H_\rho$ is an isometry,
$\xi$ is a unit vector in $H_\rho$.
In particular, $P_\rho\xi=\xi$.
Identifying $M_1(\mathbb C)$ with $\mathbb C$, for every
$b\in A\rtimes_r\Gamma$ we have
\[
\begin{aligned}
\mu_{\eta,V}(b)
=\langle\Theta_\eta(b)V(1),V(1)\rangle=\langle P_\rho\pi_\eta(b)\xi,\xi\rangle=\langle\pi_\eta(b)\xi,P_\rho\xi\rangle=\langle\pi_\eta(b)\xi,\xi\rangle.
\end{aligned}
\]
Since $\xi_\eta$ and $\xi$ are unit vectors,
\cite{Murphy1990}*{Theorem~5.2.2} gives a self-adjoint
$w\in\pi_\rho(A)$ such that $e^{iw}\xi_\eta=\xi$.
Choose $a\in A$ with $\pi_\rho(a)=w$ and let
$h=(a+a^*)/2$. Then $h=h^*$ and $\pi_\rho(h)=w$,
so $u=e^{ih}$ is a unitary in $A$ satisfying
$\pi_\rho(u)\xi_\eta=\xi$.
Therefore, for $b\in A\rtimes_r\Gamma$,
\[
\mu_{\eta,V}(b)
=\langle\pi_\eta(b)\pi_\eta(u)\xi_\eta,
\pi_\eta(u)\xi_\eta\rangle
=\eta(u^*bu).
\]
The map $b\mapsto u^*bu$ is $T_\mu$ for
$\mu=((u^*,e))\in\Prob_f(\Gamma,A)$.
Since the set $K_{\phi_0}^{\Prob_f(\Gamma,A)}$ is invariant under
postcomposition by these averaging maps, it follows that
$R_1^\eta\subseteq K_{\phi_0}^{\Prob_f(\Gamma,A)}$ for every such
$\eta$. It remains to justify the passage from these maps to the
first level of their matrix convex hull.
If $\mu_{\eta,V}\in R_n^\eta$ and
$\gamma\in M_{n,1}\setminus\{0\}$, then
\[
\gamma^*\mu_{\eta,V}(\cdot)\gamma
=\|\gamma\|^2
\left\langle
\Theta_\eta(\cdot)\frac{V\gamma}{\|\gamma\|},
\frac{V\gamma}{\|\gamma\|}
\right\rangle.
\]
Thus every scalar compression of an element of $R_n^\eta$
is a positive scalar multiple of an element of $R_1^\eta$.
A level-one matrix convex combination is therefore an
ordinary convex combination of states in the families
$R_1^\eta$, possibly for different $\eta$.
Since $K_{\phi_0}^{\Prob_f(\Gamma,A)}$ is weak$^*$ closed and convex,
we conclude that
$\mathcal K_1\subseteq K_{\phi_0}^{\Prob_f(\Gamma,A)}$.

We finally verify equivariance.
For $g\in\Gamma$, let
$\beta_g=\operatorname{Ad}(\lambda_g)$ on
$A\rtimes_r\Gamma$, and use the convention
$(g\cdot\mu)(b)=\mu(\beta_{g^{-1}}(b))$ from
\cite{KS19}*{Section~7.2}.
If $\mu_{\eta,V}\in R_n^\eta$, let
$\eta_g=\eta\circ\beta_{g^{-1}}$.
The automorphism $\beta_{g^{-1}}$ is $T_\mu$ for
$\mu=((1,g^{-1}))\in\Prob_f(\Gamma,A)$, so
$\eta_g\in K_{\phi_0}^{\Prob_f(\Gamma,A)}$, and $\eta_g|_A$ is pure.
Since $R_n^{\eta_g}$ is independent of the chosen GNS
realization, we may use
$(\pi_\eta\circ\beta_{g^{-1}},H_\eta,\xi_\eta)$
as a GNS triple for $\eta_g$.
Since $\beta_{g^{-1}}(A)=A$, its cyclic $A$-subspace is
again $H_\rho$, with the same orthogonal projection.
In this realization,
$\Theta_{\eta_g}=\Theta_\eta\circ\beta_{g^{-1}}$.
Consequently,
$g\cdot\mu_{\eta,V}=\mu_{\eta_g,V}\in R_n^{\eta_g}$.
Thus the family of all these maps, and hence $\mathcal K$,
is $\Gamma$-invariant.
The restriction map $r$ is equivariant, continuous,
matrix affine, and surjective.

By \cite{KS19}*{Corollary~7.7}, there is an invariant
closed matrix convex subset $\mathcal L\subseteq\mathcal K$
such that $(\mathcal L,r|_{\mathcal L})$ is a boundary for
$(\operatorname{MS}(A),\Gamma)$.
The canonical boundary is
$\mathcal K_E=\{\nu\circ E:\nu\in\operatorname{MS}(A)\}$.
Its first level is $(\mathcal K_E)_1=\SE$.
On the other hand,
$\mathcal L_1\subseteq\mathcal K_1
\subseteq K_{\phi_0}^{\Prob_f(\Gamma,A)}$, while
$K_{\phi_0}^{\Prob_f(\Gamma,A)}\cap\SE=\varnothing$.
Hence $\mathcal L\neq\mathcal K_E$.
By \cite{KS19}*{Theorem~7.8}, boundaries for
$(\operatorname{MS}(A),\Gamma)$ inside
$\operatorname{MS}(A\rtimes_r\Gamma)$ correspond bijectively
to $\Gamma$-equivariant pseudo-expectations for
$A\rtimes_r\Gamma$, the canonical boundary $\mathcal K_E$
corresponding to $E$.
Let $\Theta\colon A\rtimes_r\Gamma\to I_\Gamma(A)$,
$\Theta|_A=\id_A$, be the equivariant unital completely
positive map corresponding to $\mathcal L$.
Since $\mathcal L\neq\mathcal K_E$, injectivity of this
correspondence gives $\Theta\neq E$.
This contradicts uniqueness.
Therefore $A\rtimes_r\Gamma$ has GPAP.
\end{proof}
We now proceed
to prove Theorem~\ref{thm:mainlifting} from the introduction. When the coefficients of the averaging unitaries are commutative, the general principle as established in
\cite{AmrutamKalantar2020}*{Lemma~2.1} and
\cite{AmrutamUrsu2022}*{Lemma~4.1} is that an averaging procedure
which makes the nontrivial group unitaries small also makes every
element with zero canonical conditional expectation small. For the equivariant injective envelope, as we prove below, the averages
upstairs can still be taken with coefficients in $A$. We denote by $\mathsf S_{E_{I_\Gamma(A)}}$ the set $\{\omega\circ E_{I_\Gamma(A)}:\omega\in S(I_\Gamma(A))\}$.
\begin{theorem}
\label{thm:ncpap-full-envelope-lifting}
If $A\rtimes_r\Gamma$ has GPAP, then
$I_\Gamma(A)\rtimes_r\Gamma$ has relative GPAP with coefficients
in $A$. The averages can be chosen simultaneously on every
finite set. Moreover,
$\mathsf S_{E_{I_\Gamma(A)}}\subseteq K_{\varphi}^{\Prob_f(\Gamma,A)}$ for every
$\varphi\in S(I_\Gamma(A)\rtimes_r\Gamma)$.
\end{theorem}

\begin{proof}
Let us write $E=E_{I_\Gamma(A)}$ for ease of notation. Fix $y_1,\ldots,y_n\in I_\Gamma(A)\rtimes_r\Gamma$
and $\varepsilon>0$, and let $x_i=y_i-E(y_i)\in\ker E$. We need to find
$\mu\in\Prob_f(\Gamma,A)$ with $\|T_\mu(x_i)\|<\varepsilon$ for every $1\leq i\leq n$.

Now, let
$B=I_\Gamma(A)$. In view of Remark~\ref{rem:convex-semigroup} and Proposition~\ref{prop:orbit-intersection}, it suffices to show that
$\mathsf S_{E_{I_\Gamma(A)}}\subseteq K_{\varphi}^{\Prob_f(\Gamma,A)}$ for every
$\varphi\in S(I_\Gamma(A)\rtimes_r\Gamma)$.
Fix such a state $\varphi$. Since $A\rtimes_r\Gamma$ has GPAP, Proposition~\ref{prop:orbit-intersection},
applied for $A$, gives a net
$(\mu_\nu)_\nu$ in $\Prob_f(\Gamma,A)$ with $\|T_{\mu_\nu}(x)\|\to0$ for every
$x\in\ker E_A$. The same formulas define these maps on
$I_\Gamma(A)\rtimes_r\Gamma$. Passing to a subnet, let
$\eta_0=\lim_\nu\varphi\circ T_{\mu_\nu}$ in the weak$^*$ topology, so that
$\eta_0\in K_{\varphi}^{\Prob_f(\Gamma,A)}$. For $d\in A\rtimes_r\Gamma$ we have
$T_{\mu_\nu}(d)-T_{\mu_\nu}(E_A(d))=T_{\mu_\nu}(d-E_A(d))\to0$ in norm, and therefore
\[
 \eta_0(d)=\lim_\nu\varphi\bigl(T_{\mu_\nu}(E_A(d))\bigr)=\eta_0(E_A(d)).
\]
Thus $\eta_0|_{A\rtimes_r\Gamma}=\psi\circ E_A$, where $\psi=\eta_0|_A$.

Choose a pure state $\rho\in S(A)$. By
Lemma~\ref{lem:elementary-in-T}, there is a net $(\mu_\kappa)_\kappa$ in
$\Prob_f(\Gamma,A)$ with $\|T_{\mu_\kappa}(a)-\rho(a)1\|\to0$ for every
$a\in A$; since the coefficients of $\mu_\kappa$ lie in $A$, the maps
$T_{\mu_\kappa}$ are also defined on
$I_\Gamma(A)\rtimes_r\Gamma$. For $d\in A\rtimes_r\Gamma$,
compatibility with $E_A$ gives
$(\eta_0T_{\mu_\kappa})(d)=\psi\bigl(T_{\mu_\kappa}(E_A(d))\bigr)
 \longrightarrow\rho(E_A(d))$.
Each $\eta_0\circ T_{\mu_\kappa}$ lies in $K_{\varphi}^{\Prob_f(\Gamma,A)}$, which is
invariant under postcomposition by $T_\mu$ for $\mu\in\Prob_f(\Gamma,A)$. A weak$^*$ cluster
point $\eta$ of this net therefore belongs to $K_{\varphi}^{\Prob_f(\Gamma,A)}$
and satisfies $\eta|_{A\rtimes_r\Gamma}=\rho\circ E_A$.

Let $(\pi_\eta,\mathcal H,\Omega)$ be the GNS representation of $\eta$
and let $\mathcal H_\rho=\overline{\pi_\eta(A)\Omega}$. This subspace
reduces $\pi_\eta(A)$. Since $\eta|_A=\rho$, the restriction $\pi_\rho$
of $\pi_\eta|_A$ to $\mathcal H_\rho$, with cyclic vector $\Omega$, is
the GNS representation of $\rho$; hence it is irreducible. Let $p_\rho$
be the orthogonal projection onto $\mathcal H_\rho$ and define the
u.c.p.~map
\[
 Q\colon I_\Gamma(A)\rtimes_r\Gamma\longrightarrow B(\mathcal H_\rho),
~ Q(y)=p_\rho\pi_\eta(y)|_{\mathcal H_\rho},
\]
so that $Q|_A=\pi_\rho$. Since
$\eta|_{A\rtimes_r\Gamma}=\rho\circ E_A$, for $g\neq e$ and $a,b\in A$ we have
\[
 \begin{aligned}
 \bigl\langle Q(\lambda_g)\pi_\rho(b)\Omega,\pi_\rho(a)\Omega\bigr\rangle
 &=\eta(a^*\lambda_gb)\\
 &=\eta\bigl(a^*\alpha_g(b)\lambda_g\bigr)\\
 &=\rho\bigl(E_A(a^*\alpha_g(b)\lambda_g)\bigr)=0 .
 \end{aligned}
\]
Since $\pi_\rho(A)\Omega$ is dense in $\mathcal H_\rho$, it follows that
$Q(\lambda_g)=0$ for every $g\neq e$.

Equip $\ell^\infty(\Gamma,B(\mathcal H_\rho))$ with the left translation
action $(s\cdot m)(t)=m(s^{-1}t)$, and define
\[
 q\colon I_\Gamma(A)\rtimes_r\Gamma\longrightarrow
 \ell^\infty(\Gamma,B(\mathcal H_\rho)),
 ~ q(y)(t)=Q(\lambda_t^*y\lambda_t).
\]
The map $q$ is u.c.p. Since $\lambda_t^*\lambda_s=\lambda_{t^{-1}s}=\lambda_{s^{-1}t}^*$, it
follows that
\[
 q(\lambda_sy\lambda_s^*)(t)
 =Q(\lambda_{s^{-1}t}^*\,y\,\lambda_{s^{-1}t})=q(y)(s^{-1}t).
\]
Therefore, $q$ is equivariant with respect to the conjugation action.
For $a\in A$ we have $q(a)(t)=\pi_\rho(\alpha_{t^{-1}}(a))$, so $q|_A$ is
a unital $*$-homomorphism. Its kernel
$\bigcap_{t\in\Gamma}\ker(\pi_\rho\circ\alpha_{t^{-1}})$ is a proper
$\Gamma$-invariant ideal of $A$, hence zero by $\Gamma$-simplicity. Thus
$q(a)\mapsto a$ is a well-defined equivariant u.c.p.~map
$q(A)\to A\subseteq I_\Gamma(A)$. By $\Gamma$-injectivity of $I_\Gamma(A)$, this map extends to an
equivariant u.c.p.~map
$P\colon\ell^\infty(\Gamma,B(\mathcal H_\rho))\to I_\Gamma(A)$. The map
$(Pq)|_{I_\Gamma(A)}$ is equivariant, u.c.p., and fixes $A$ pointwise, so
equivariant rigidity gives $(Pq)|_{I_\Gamma(A)}=\id_{I_\Gamma(A)}$ (see
\cite{GU25}*{Theorem~2.8 and Definition~2.9}).

If $g\neq e$, then
$t^{-1}gt\neq e$ for every $t\in\Gamma$, so
$q(\lambda_g)(t)=Q(\lambda_{t^{-1}gt})=0$ and hence $q(\lambda_g)=0$.
Since $Pq$ fixes $I_\Gamma(A)$ pointwise, $I_\Gamma(A)$ lies in the
multiplicative domain of $Pq$, whence
$(Pq)(b\lambda_g)=b\,(Pq)(\lambda_g)=0$ for $b\in I_\Gamma(A)$ and
$g\neq e$. Consequently, it follows that $Pq=E$.

For $t\in\Gamma$ and a unit vector $\xi\in\mathcal H_\rho$, let
$\omega_{t,\xi}(m)=\langle m(t)\xi,\xi\rangle$. We claim that
$\omega_{t,\xi}\circ q\in K_{\varphi}^{\Prob_f(\Gamma,A)}$. Since $\pi_\rho$ is
irreducible and $\Omega,\xi$ are unit vectors in $\mathcal H_\rho$,
\cite{Murphy1990}*{Theorem~5.2.2},
together with the self-adjoint lifting argument in the proof of
Theorem~\ref{thm:unique-pseudo-expectation}, gives a unitary
$u\in A$ with $\pi_\rho(u)\Omega=\xi$. Hence, for
$y\in I_\Gamma(A)\rtimes_r\Gamma$,
\[
 (\omega_{t,\xi}\circ q)(y)
 =\bigl\langle\pi_\eta(\lambda_t^*y\lambda_t)\pi_\eta(u)\Omega,
 \pi_\eta(u)\Omega\bigr\rangle
 =\eta(u^*\lambda_t^*y\lambda_tu).
\]
The map $y\mapsto u^*\lambda_{t^{-1}}y\lambda_{t^{-1}}^*u$ is $T_\mu$ for
$\mu=((u^*,t^{-1}))\in\Prob_f(\Gamma,A)$, since $u^*u=1$. The claim now
follows from the fact that $\eta\in K_{\varphi}^{\Prob_f(\Gamma,A)}$ and its invariance.

For $m=m^*\in\ell^\infty(\Gamma,B(\mathcal H_\rho))$ we have $m\leq r1$
if and only if $\langle m(t)\xi,\xi\rangle\leq r$ for every $t\in\Gamma$
and every unit vector $\xi$. By Lemma~\ref{lem:density}, the convex hull
of the states $\omega_{t,\xi}$ is therefore weak$^*$ dense in
$S(\ell^\infty(\Gamma,B(\mathcal H_\rho)))$.

Finally, fix $\omega\in S(I_\Gamma(A))$. The state $\omega\circ P$ is a
weak$^*$ limit of convex combinations $c_\kappa$ of the
$\omega_{t,\xi}$. Composition with $q$ is weak$^*$ continuous, so
$c_\kappa\circ q\to\omega\circ P\circ q$; each $c_\kappa\circ q$ lies in
the convex weak$^*$ closed set $K_{\varphi}^{\Prob_f(\Gamma,A)}$.
Since $Pq=E$, we obtain
$\omega\circ E=\omega\circ P\circ q\in K_{\varphi}^{\Prob_f(\Gamma,A)}$.
Thus $\mathsf S_{E_{I_\Gamma(A)}}\subseteq K_{\varphi}^{\Prob_f(\Gamma,A)}$,
which proves the theorem.

\end{proof}
We now combine the preceding results to prove Theorem~\ref{thm:mainaveraging} from the introduction.
\begin{theorem}
\label{thm:main}
For a unital $\Gamma$-simple $\Gamma$-$C^*$-algebra $A$, the following
are equivalent:
\begin{enumerate}[label=\textup{(\roman*)},leftmargin=*]
\item $A\rtimes_r\Gamma$ has GPAP.
\item $I_\Gamma(A)\rtimes_r\Gamma$ has relative GPAP with coefficients in $A$.
\item $I_\Gamma(A)\rtimes_r\Gamma$ has GPAP.
\item $E_{I_\Gamma(A)}$ is the unique equivariant conditional expectation
$I_\Gamma(A)\rtimes_r\Gamma\to I_\Gamma(A)$.
\item The canonical map $A\rtimes_r\Gamma\to I_\Gamma(A)$ is the unique
equivariant pseudo-expectation.
\end{enumerate}
The norm-averaging statements can all be imposed simultaneously
on finite sets.
\end{theorem}
\begin{proof}
We prove the cycle
\textup{(i)}$\Rightarrow$\textup{(ii)}$\Rightarrow$\textup{(iii)}$\Rightarrow$\textup{(iv)}$\Rightarrow$\textup{(v)}$\Rightarrow$\textup{(i)}.
Theorem~\ref{thm:ncpap-full-envelope-lifting} proves
\textup{(i)}$\Rightarrow$\textup{(ii)}, and the inclusion
$A\subseteq I_\Gamma(A)$ gives
\textup{(ii)}$\Rightarrow$\textup{(iii)}.

\noindent\textup{(iii)}$\Rightarrow$\textup{(iv)}.
We use the range of the difference of two expectations to distinguish
their state orbits. The same argument works for any unital
$\Gamma$-simple coefficient algebra with GPAP.
Suppose, towards a contradiction, that
$P\colon I_\Gamma(A)\rtimes_r\Gamma\to I_\Gamma(A)$ is an equivariant conditional expectation
different from $E_{I_\Gamma(A)}$. Let 
$\mathcal R=(P-E_{I_\Gamma(A)})(I_\Gamma(A)\rtimes_r\Gamma),~J=\overline{\mathcal R}^{\|\cdot\|}\subseteq I_\Gamma(A)$.
 Both maps are $I_\Gamma(A)$-bimodular. Thus $\mathcal R$ is an $I_\Gamma(A)$-bimodule;
it is also linear, self-adjoint, and $\Gamma$-invariant.
Consequently, $J$ is a nonzero closed $\Gamma$-invariant two-sided
ideal of $I_\Gamma(A)$. Lemma~\ref{lem:envelope-gamma-simple} gives $J=I_\Gamma(A)$.

Consider the weak$^*$ compact sets 
${\mathsf S}_P=\{\rho\circ P:\rho\in S(I_\Gamma(A))\}$, and ${\mathsf S}_{E_{I_\Gamma(A)}}$.
 They are disjoint. Indeed, if $\rho \circ P=\sigma\circ E_{I_\Gamma(A)}$, restriction
to $I_\Gamma(A)$ gives $\rho=\sigma$. Hence $\rho$ vanishes on
$(P-E_{I_\Gamma(A)})(I_\Gamma(A)\rtimes_r\Gamma)$ and therefore on $J=I_\Gamma(A)$, contradicting
$\rho(1)=1$.

For $\mu\in\Prob_f(\Gamma,I_\Gamma(A))$, bimodularity and equivariance imply $P\circ T_\mu=(T_\mu|_{I_\Gamma(A)})\circ P.$ Therefore $\mathsf S_P$ is invariant under $I_\Gamma(A)$-coefficient averaging.
Fix $\rho\in S(I_\Gamma(A))$ and let $\varphi=\rho \circ P$. Then $K_{\varphi}^{\Prob_f(\Gamma,I_{\Gamma}(A))}\subseteq\mathsf S_P.$ This orbit closure cannot meet $\mathsf S_{E_{I_\Gamma(A)}}$,  contradicting
GPAP and Proposition~\ref{prop:orbit-intersection} with coefficients in $I_\Gamma(A)$.
The same argument applies verbatim with $\Prob_f(\Gamma,I_\Gamma(A))$
replaced by $\Prob_f(\Gamma,A)$, since $P\circ T_\mu=(T_\mu|_{I_\Gamma(A)})\circ P$
also holds for $\mu\in\Prob_f(\Gamma,A)$; thus \textup{(ii)} alone already
implies \textup{(iv)}.

\noindent\textup{(iv)}$\Leftrightarrow$\textup{(v)}.
We recall the correspondence from~\cite{GU25}*{Proposition~3.2}.
An equivariant pseudo-expectation $\Phi\colon A\rtimes_r\Gamma\to I_\Gamma(A)$ extends,
by equivariant injectivity, to an equivariant u.c.p.~map
$\widetilde\Phi\colon I_\Gamma(A)\rtimes_r\Gamma\to I_\Gamma(A)$.
Its restriction to $I_\Gamma(A)$ fixes $A$, so rigidity gives
$\widetilde\Phi|_{I_\Gamma(A)}=\id_{I_\Gamma(A)}$. It is therefore a conditional expectation.

Conversely, the restriction of any equivariant conditional
expectation is a pseudo-expectation. If two such conditional
expectations agree on $A\rtimes_r\Gamma$, they agree on every $\lambda_g$.
Their $I_\Gamma(A)$-bimodularity makes them agree on every $b\lambda_g$,
$b\in I_\Gamma(A)$, and density gives equality. This also proves uniqueness
of the extension.

\noindent\textup{(v)}$\Rightarrow$\textup{(i)}.
This follows from Theorem~\ref{thm:unique-pseudo-expectation} and
closes the cycle.

The finite-set assertions follow from
Lemma~\ref{lem:centered-finite-averaging}.
\end{proof}

\begin{remark}
\label{rem:downstairs-ce-distinction}
The equivalent conditions above imply that $E_A$ is the unique
equivariant conditional expectation from $A\rtimes_r\Gamma$ onto $A$.
The converse is false, even when $A=\mathbb C$.
Indeed, let $\Gamma$ be a countable group with trivial amenable radical
which is not $C^*$-simple, as in~\cite{LeBoudec2017}*{Theorem~C}.
Equivariant conditional expectations from $C_r^*(\Gamma)$ onto
$\mathbb C$ are exactly tracial states. There is a unique such
expectation by~\cite{BKKO17}*{Theorem~1.3}, whereas GPAP would
imply simplicity by Corollary~\ref{cor:ncpap-all-intermediates-simple}.
The codomain $I_\Gamma(A)$ in condition~\textup{(v)} is therefore
essential.
\end{remark}

\section{Simplicity}
\label{sec:simplicity}

The relative averages preserve every intermediate algebra and each of
its ideals. This gives the following application.

\begin{corollary}
\label{cor:ncpap-all-intermediates-simple}
Assume any of the equivalent conditions of Theorem~\ref{thm:main}.
Then every intermediate $C^*$-algebra $A\rtimes_r\Gamma\subseteq B\subseteq I_\Gamma(A)\rtimes_r\Gamma$ is simple. 
\end{corollary}
\begin{proof}
Let $J$ be a proper closed two-sided ideal of $B$.
Choose a state of the nonzero unital quotient $B/J$, pull it back
to $B$, and extend it to a state $\varphi\in S(I_\Gamma(A)\rtimes_r\Gamma)$.
Thus $\varphi(J)=0$.

For every $\mu\in\Prob_f(\Gamma,A)$, the map $T_\mu$ preserves $B$ and $J$, because its
multipliers $a_j\lambda_{t_j}$ belong to $A\rtimes_r\Gamma\subseteq B$.
In particular, $a_j\lambda_{t_j}J\lambda_{t_j}^*a_j^*\subseteq J.$ It follows that every state in $K_{\varphi}^{\Prob_f(\Gamma,A)}$ vanishes on $J$.
However, Theorem~\ref{thm:ncpap-full-envelope-lifting} gives
$\omega E_{I_\Gamma(A)}\in K_{\varphi}^{\Prob_f(\Gamma,A)}$ for every $\omega\in S(I_\Gamma(A))$.
Consequently, for $j\in J$, $\omega(E_{I_\Gamma(A)}(j^*j))=0~(\omega\in S(I_\Gamma(A))).$ States separate positive elements, so $E_{I_\Gamma(A)}(j^*j)=0$.
Faithfulness of $E_{I_\Gamma(A)}$ gives $j=0$. Therefore $J=0$ and $B$ is simple.
\end{proof}
\begin{remark}
\label{rem:hamana-bryder-simplicity}
The simplicity conclusion also follows from simplicity downstairs
alone. By \cite{Hamana1985}*{Theorem~3.4}, we have the canonical inclusions
$A\rtimes_r\Gamma\subseteq I_\Gamma(A)\rtimes_r\Gamma
\subseteq I(A\rtimes_r\Gamma)$; see also
\cite{GU25}*{Theorem~2.12}.
Every intermediate algebra between a simple unital algebra and its
ordinary injective envelope is simple (see also~\cite{Bryder2022Injective}*{Theorem~3.2}).
\end{remark}
The intersection property is equivalent to faithfulness of every
equivariant pseudo-expectation, by~\cite{KS19}*{Theorem~6.6};
see also~\cite{GU25}*{Proposition~3.3 and the discussion following it}.
For a $\Gamma$-simple coefficient algebra, the intersection
property is equivalent to simplicity of its reduced crossed product.
Theorem~\ref{thm:main}, on the other hand, concerns uniqueness.
The intersection-property criterion does not by itself give that
uniqueness. 

For the rest of this section we assume that
$A\rtimes_r\Gamma$ is simple, and we consider the two additional
conditions studied in \cite{KS19} and \cite{GU25}. The first one is about the action having
\emph{vanishing obstruction} in the sense of
\cite{KS19}*{Definition~8.3} (Condition~\textup{(V)}), and the second is about the group
$\Gamma$ being FC-hypercentral (Condition~\textup{(F)}).  Under this standing assumption, the action on $A$ has the intersection
property. By \cite{KS19}*{Theorem~5.5}, the action on $I_\Gamma(A)$ also
has the intersection property. Since $I_\Gamma(A)$ is $\Gamma$-simple
by Lemma~\ref{lem:envelope-gamma-simple}, it follows that
$I_\Gamma(A)\rtimes_r\Gamma$ is simple as well.

Recall that $\operatorname{FC}(\Gamma)$ is the normal subgroup of
elements with finite conjugacy class. We first record what these
classes detect directly.

\begin{lemma}
\label{lem:restored-finite-conjugacy}
Assume that $I_\Gamma(A)\rtimes_r\Gamma$ is simple.
If $P:I_\Gamma(A)\rtimes_r\Gamma\to I_\Gamma(A)$ is an
equivariant conditional expectation, then
$P(\lambda_s)=0$ for every
$s\in\operatorname{FC}(\Gamma)\setminus\{e\}$.
\end{lemma}

\begin{proof}
Let $x_t=P(\lambda_t)$ and let $\mathcal C$ be the conjugacy
class of $s$. By~\cite{GU25}*{Propositions~3.4 and~3.6},
$\sum_{t\in\mathcal C}x_t^*\lambda_t$ is central in
$I_\Gamma(A)\rtimes_r\Gamma$. A simple unital
$C^*$-algebra has trivial centre, so this element is a scalar multiple
of $1$, and its canonical expectation is zero. Hence it is zero,
and comparison of Fourier coefficients gives $x_s=0$.
\end{proof}

For an FC-hypercentral group, the preceding vanishing can be
propagated along the FC-central series. The following argument appeared in an earlier version of
Geffen--Ursu~\cite{GU25}. We retain the argument here because we need uniqueness
of the equivariant pseudo-expectation.
\begin{proposition}
\label{prop:restored-fc-hypercentral}
Assume that $A\rtimes_r\Gamma$ is simple.
Condition~\textup{(F)} implies uniqueness of the equivariant
pseudo-expectation.
\end{proposition}
\begin{proof}
Let $(F_\xi)$ be the FC-central series of $\Gamma$, with
$F_0=\{e\}$, $F_{\xi+1}/F_\xi=\operatorname{FC}(\Gamma/F_\xi)$,
and $F_\eta=\bigcup_{\xi<\eta}F_\xi$ at limit ordinals.
Each $F_\xi$ is a normal subgroup of $\Gamma$.
Extend an equivariant pseudo-expectation to an equivariant
conditional expectation
$P:I_\Gamma(A)\rtimes_r\Gamma\to I_\Gamma(A)$, and write
$x_s=P(\lambda_s)$ and $z_s=x_s^*x_s$. The intertwining relation in \cite{GU25}*{Proposition~3.4} gives
$z_sb=x_s^*\alpha_s(b)x_s=bx_s^*x_s=bz_s$ for every $b\in I_\Gamma(A)$,
so $z_s\in Z(I_\Gamma(A))$.
The action on the spectrum of this center is minimal by
Lemma~\ref{lem:envelope-gamma-simple}. FC-hypercentral
groups are amenable, and amenability of $\Gamma$ therefore gives a
faithful invariant state $\omega$ on the center. Lemma~\ref{lem:restored-finite-conjugacy} gives $x_s=0$ on
$F_1\setminus\{e\}$. Suppose that $x_s=0$ on
$F_\xi\setminus\{e\}$, where $\xi\geq1$, and take
$h\in F_{\xi+1}\setminus F_\xi$. The conjugacy class of $h$
is infinite because $h\notin F_1$, but meets only finitely many
$F_\xi$-cosets because $hF_\xi$ has finite conjugacy class in
$\Gamma/F_\xi$. Hence one such coset contains infinitely many conjugates
of $h$. For any $k$, choose distinct conjugates
$h_i=g_ihg_i^{-1}$, $1\leq i\leq k$, in that coset.
Normality of $F_\xi$ gives $h_ih_j^{-1}\in F_\xi\setminus\{e\}$
for $i\neq j$, so $P(\lambda_{h_ih_j^{-1}})=0$ by the induction hypothesis.
Writing $a_i=x_{h_i}$ and
$W=(1,\lambda_{h_1},\ldots,\lambda_{h_k})^{\mathsf T}$,
complete positivity gives
\[
 P^{(k+1)}(WW^*)=
 \begin{pmatrix}
 1&a_1^*&\cdots&a_k^*\\
 a_1&1&&0\\
 \vdots&&\ddots&\\
 a_k&0&&1
 \end{pmatrix}\geq0.
\]
Multiplication on the right by $(-1,a_1,\ldots,a_k)^{\mathsf T}$
and on the left by its adjoint gives
$1-\sum_{i=1}^k a_i^*a_i\geq0$.
Equivariance gives $a_i=\alpha_{g_i}(x_h)$ and hence
$a_i^*a_i=\alpha_{g_i}(z_h)$. Applying the invariant state
$\omega$ yields $k\omega(z_h)\leq1$. Since $k$ is arbitrary
and $\omega$ is faithful, $z_h=0$ and hence $x_h=0$.
Limit stages follow by taking unions. Transfinite induction
therefore gives $x_s=0$ for every $s\neq e$, so $P$ is canonical.
\end{proof}
\begin{theorem}
\label{thm:restored-three-conditions}
Under the standing simplicity assumptions,
each of conditions~\textup{(F)} and~\textup{(V)}
implies all the equivalent conditions of Theorem~\ref{thm:main}.
In particular, averaging on $I_\Gamma(A)\rtimes_r\Gamma$ can
be performed with coefficients in $A$.
\end{theorem}

\begin{proof}
Condition~\textup{(F)} is covered by
Proposition~\ref{prop:restored-fc-hypercentral}. Under~\textup{(V)},
simplicity of $A\rtimes_r\Gamma$ gives the intersection property,
and \cite{KS19}*{Theorem~9.3} implies that the action on
$I_\Gamma(A)$ is properly outer. By
\cite{KS19}*{Theorem~6.4}, the equivariant pseudo-expectation is
unique. Theorem~\ref{thm:main} then gives the averaging conclusions.
Conversely, GPAP implies simplicity of $A\rtimes_r\Gamma$ by
Corollary~\ref{cor:ncpap-all-intermediates-simple}. Thus, under either
\textup{(F)} or \textup{(V)}, simplicity of $A\rtimes_r\Gamma$ is equivalent
to GPAP.
\end{proof}
We finish this section with an example for which both conditions \textup{(F)} and \textup{(V)} fail and yet GPAP holds.
\begin{example}
Let $A=M_2(\mathbb C)$, let $K=\mathbb Z_2^2$, and let
\[
 U=\begin{pmatrix}1&0\\0&-1\end{pmatrix},~
 V=\begin{pmatrix}0&1\\1&0\end{pmatrix},~
 v_{(r,s)}=U^rV^s.
\]
Since $UV=-VU$, the maps $\beta_k=\operatorname{Ad}v_k$ define
an action of $K$ on $A$. Let $\Gamma=\mathbb F_2\times K$ act
by $\alpha_{(h,k)}=\beta_k$. This is the action in
\cite{GU25}*{Example~8.4}, extended trivially over $\mathbb F_2$.
The algebra $A$ is simple, hence $\Gamma$-simple. Condition~\textup{(F)} fails since the FC-center of $\Gamma$ is
$\{e\}\times K$, and the quotient $\mathbb F_2$ has trivial
FC-center, so the FC-central series never reaches $\Gamma$.
Condition~\textup{(V)} also fails. Indeed, $I(A)=A$ and every
$\alpha_g$ is inner, so every projection $p_g$ of
\cite{KS19}*{Definition~8.3} equals $1$, and a partial
$*$-representation is then a unitary representation; vanishing
obstruction would therefore give a unitary representation implementing
the action. Its restriction
to $K$ would give commuting implementers of
$\operatorname{Ad}U$ and $\operatorname{Ad}V$. These must be
scalar multiples of $U$ and $V$, respectively, and therefore
anticommute, contradicting the fact that they commute. Hence
\textup{(V)} fails.

For $\ell=(r,s)$ and $k=(r',s')$, let
$\chi_\ell(k)=(-1)^{rs'-sr'}$. The unitaries
$v_\ell^*\lambda_{(e,\ell)}$ commute with $A$, and covariance gives
\[
 v_\ell^*\lambda_{(e,\ell)}\lambda_{(h,k)}
 \lambda_{(e,\ell)}^*v_\ell
 =\chi_\ell(k)\lambda_{(h,k)}.
\]
Since $\frac14\sum_{\ell\in K}\chi_\ell(k)=\delta_{k,e}$, the map
\[
 T_K(y)=\frac14\sum_{\ell\in K}
 v_\ell^*\lambda_{(e,\ell)}y\lambda_{(e,\ell)}^*v_\ell
\]
is $T_{\mu_K}$ for
$\mu_K=((v_\ell^*/2,(e,\ell)))_{\ell\in K}\in\Prob_f(\Gamma,A)$, and
$T_K(a\lambda_{(h,k)})=\delta_{k,e}a\lambda_{(h,e)}$ for $a\in A$.
Thus four terms kill the nontrivial $K$-coordinate.

Let $x=\sum_{h,k}a_{h,k}\lambda_{(h,k)}$ be a finite sum with
$E_A(x)=0$, and let $\varepsilon>0$. Then $a_{e,e}=0$ and
$T_K(x)=\sum_{h\neq e}a_{h,e}\lambda_{(h,e)}$.
The canonical copy of $C_r^*(\mathbb F_2)$ in $A\rtimes_r\Gamma$ is
isometric. By \cite{Powers1975}, there are $t_1,\ldots,t_m\in\mathbb F_2$
such that, for every $h\neq e$ occurring in $T_K(x)$,
\[
 \left\|\frac1m\sum_{j=1}^m
 \lambda_{(t_jht_j^{-1},e)}\right\|
 <\frac{\varepsilon}{1+\sum_{h\neq e}\|a_{h,e}\|}.
\]
Since $\mathbb F_2$ acts trivially on $A$, the composite average
\[
 \begin{aligned}
 T(y)&=\frac1m\sum_{j=1}^m
 \lambda_{(t_j,e)}T_K(y)\lambda_{(t_j,e)}^*\\
 &=\frac1{4m}\sum_{j=1}^m\sum_{\ell\in K}
 v_\ell^*\lambda_{(t_j,\ell)}y\lambda_{(t_j,\ell)}^*v_\ell.
 \end{aligned}
\]
It satisfies
\[
 \begin{aligned}
 \|T(x)\|
 &\leq\sum_{h\neq e}\|a_{h,e}\|
 \left\|\frac1m\sum_{j=1}^m\lambda_{(t_jht_j^{-1},e)}\right\|\\
 &\leq\frac{\varepsilon\sum_{h\neq e}\|a_{h,e}\|}
 {1+\sum_{h\neq e}\|a_{h,e}\|}<\varepsilon.
 \end{aligned}
\]
Its coefficients
$v_\ell^*/(2\sqrt m)$ lie in $A$ and satisfy
$\sum_{j,\ell}v_\ell^*v_\ell/(4m)=1$. Density and contractivity
give GPAP for $A\rtimes_r\Gamma$.
Theorem~\ref{thm:ncpap-full-envelope-lifting} gives relative
GPAP on $I_\Gamma(A)\rtimes_r\Gamma$ with coefficients in $A$,
and Corollary~\ref{cor:ncpap-all-intermediates-simple} gives
simplicity of every intermediate algebra.
\end{example}
 
\subsection*{Acknowledgements}
This work was completed during the second author's visit to the University of Victoria. We sincerely thank the University of Victoria for its warm hospitality during this period.
\bibliographystyle{amsalpha}
\bibliography{ncpap_final_2026-09-19}
\end{document}